\documentclass[11pt]{amsart}
\usepackage{amsmath,amssymb,hyperref,geometry,graphicx,tikz}

\newcommand{\C}{\mathbb{C}}
\newcommand{\Q}{\mathbb{Q}}
\newcommand{\R}{\mathbb{R}}
\newcommand{\Z}{\mathbb{Z}}
\renewcommand{\O}{\mathcal{O}}
\renewcommand{\P}{\mathbb{P}}
\newcommand{\M}{\overline{\mathcal{M}}}

\renewcommand{\L}{\mathcal{L}}
\newcommand{\vir}{\mathrm{vir}}
\renewcommand{\t}{\mathbf{t}}
\newcommand{\G}{\mathcal{G}}
\newcommand{\J}{\mathcal{J}}

\newtheorem{dummy}{}[subsection]
\newtheorem{theorem}[dummy]{Theorem}

\newtheorem{proposition}[dummy]{Proposition}
\theoremstyle{definition}
\newtheorem{example}[dummy]{Example}
\newtheorem{remark}[dummy]{Remark}

\newtheorem{exercise}[dummy]{Exercise}

\begin{document}

\title{Localization and Mirror Symmetry}
\author[D.~Ross]{Dustin Ross}
\address{Department of Mathematics, San Francisco State University, 1600 Holloway Avenue, San Francisco, CA 94132, USA}
\email{rossd@sfsu.edu}

\begin{abstract}
These notes were born out of a five-hour lecture series for graduate students at the May 2018 Snowbird workshop \emph{Crossing the Walls in Enumerative Geometry}. After a short primer on equivariant cohomology and localization, we provide proofs of the genus-zero mirror theorems for the quintic threefold, first in Fan-Jarvis-Ruan-Witten theory and then in Gromov-Witten theory. We make no claim to originality, except in exposition, where special emphasis is placed on peeling away the standard technical machinery and viewing the mirror theorems as closed-formula manifestations of elementary localization recursions.
\end{abstract}

\maketitle

\tableofcontents

\section{Introduction}

The mirror theorem for the quintic threefold $Q\subset\P^4$, first conjectured in the physics literature by Candelas, de la Ossa, Green, and Parkes \cite{COGP}, reveals deep and surprising structure in the enumerative geometry of rational curves in $Q$. In the two decades since its original proof by Givental \cite{G}, the mirror theorem has undergone countless reinterpretations and generalizations. The proof we present here is an adaptation of ideas that were developed more generally in \cite{CCIT} and \cite{RR}, and they generalize naturally to prove genus-zero mirror theorems for complete intersections in toric Deligne-Mumford stacks. En route to proving the mirror theorem for the Gromov-Witten theory of $Q$, we also take a detour to prove the mirror theorem for the Fan-Jarvis-Ruan-Witten theory of $Q$, a result that was first proved by Chiodo and Ruan \cite{CR}. The combination of these two mirror theorems leads to the proof of Witten's genus-zero ``Landau-Ginzburg/Calabi-Yau correspondence'', which we do not discuss explicitly, but which is lurking in the background.

Torus localization is the primary tool used in the proofs of the genus-zero mirror theorems presented here. Indeed, a direct application of the localization theorem of Atiyah and Bott \cite{AB} leads to an explicit algorithm that computes any genus-zero Gromov-Witten invariant of $Q$. However, as is typical with such localization computations, the combinatorial complexity of the algorithm grows at an unmanageable rate, and it quickly becomes apparent that any hope to pin down a closed formula for the Gromov-Witten invariants must result from the recursive structure of the localization contributions. The main content of the proof, then, is a clever packaging of localization recursions.

While the mirror theorems are beautiful and important results in their own right, the intent of these lectures is not simply to present proofs of these two results for historical enlightenment. Rather, our goal is to help the student familiarize herself with the methods of manipulating localization relations in two different but related settings. Localization is one of only a few computational methods in the Gromov-Witten theory toolkit, and nearly all of the current approaches to understanding higher-genus Gromov-Witten and Fan-Jarvis-Ruan-Witten theory utilize clever applications of localization on certain auxiliary moduli spaces. It is in this light that we view the genus-zero mirror theorems as a useful setting through which the student can begin honing her localization skills today in order to apply them in more creative settings that push the boundaries of the field tomorrow. 

\subsection{Target audience}

These lectures are intended for advanced graduate students who have already taken a few courses in algebraic geometry. Some exposure to moduli spaces of stable curves and stable maps would certainly be useful; we suggest \cite[Chapters 22 - 25]{MS}  for the requisite background on $\M_{g,n}$ and $\M_{g,n}(X,d)$. A working knowledge of orbifold curves and line bundles on them would also be useful; a good starting place is Chapter 1 of \cite{ALR}.

\subsection{Disclaimer}

For the sake of brevity, we have chosen to make many omissions, and we have made no attempt to make this a comprehensive account of mirror symmetry. Two such accounts already exist, and we highly recommend them: \cite{CK} and \cite{MS}. Our hope for these notes is that they will serve as a more compact reference that only emphasizes the aspects of the story that are essential to the proofs.

\subsection{Acknowledgements}

Special thanks to the organizers of the Snowbird workshop \emph{Crossing the Walls in Enumerative Geometry} for the invitation to speak; to the basketball and hot tub crew for providing a good deal of evening entertainment while we were at Snowbird; and to Emily Clader for many conversations related to these lectures and for feedback on a preliminary version of these notes.

\section{Equivariant cohomology and localization}

In this lecture, we introduce equivariant cohomology by way of several instructive examples that will be important throughout the subsequent lectures. We state the localization theorem of Atiyah and Bott in the form that will be most applicable towards the mirror theorems, and we give a few simple applications of the localization theorem.

\subsection{Equivariant Cohomology}

Throughout this section, let $X$ be a smooth, projective variety, and let $T=\{(t_1,\dots,t_n)\}=(\C^*)^n$ be an algebraic torus acting on $X$. In the same way that the cohomology ring $H^*(X)=H^*(X,\Q)$ captures some of the intersection-theoretic geometry of $X$, the \emph{equivariant cohomology ring}, denoted $H_T^*(X)$, captures some of the $T$-equivariant intersection-theoretic geometry of $X$. The formal definitions are not immediately enlightening, so we choose to forego them; essentially all of the properties necessary for the proofs in these lectures can be gleaned from the following examples. For a more thorough treatment, the reader is directed to \cite{AB}.

The first example concerns the case where the geometry (and thus, the action) is trivial.

\begin{example}
Let $X=\mathrm{pt}$ be a point. Then $H_T^*(\mathrm{pt})$ is a graded polynomial ring
\[
H_T^*(\mathrm{pt})=\Q[\alpha_1,\dots,\alpha_n],
\]
where $\deg_\R(\alpha_i)=2$. One way to interpret the classes $\alpha_i$ is via Chern classes of equivariant vector bundles. It is a standard consequence of the formal definitions that
\[
\alpha_i=c_1(\O_{\alpha_i})
\]
where $\O_{\alpha_i}$ is the equivariant line bundle on $X=\mathrm{pt}$ that is geometrically trivial but with $T$-action on fibers given by $v\mapsto t_i^{-1}v$. In fact, any $T$-equivariant vector bundle on a point (i.e. $T$-representation) uniquely decomposes as a direct sum of tensor products of the line bundles $\O_{\alpha_i}$, and its Chern classes can be determined by the usual properties of Chern classes with respect to tensor products and direct sums.
\end{example}

The second important example concerns the case where the geometry is non-trivial, but the action is trivial.

\begin{example}\label{ex:trivialaction}
Let $T$ act trivially on $X$. Then
\[
H_T^*(X)=H^*(X)\otimes\Q[\alpha_1,\dots,\alpha_n].
\]
If $V$ is a(n equivariantly trivial) rank-$r$ vector bundle on $X$, then the Euler class of $V$ is the top Chern class:
\[
e(V):=c_r(V)\in H^{2r}(X)\subset H_T^{2r}(X).
\]
Imposing a $T$-action on $V$, such as $v\mapsto t^{-\alpha}v:= t_1^{-a_1}\cdots t_n^{-a_n}v$, is the same as tensoring by an equivariant line bundle: $V\otimes\O_{\alpha}$, where $\alpha=a_1\alpha_1+\cdots+a_n\alpha_n$. In this case,
\[
e(V\otimes\O_\alpha)=c_r(V)+c_{r-1}(V)\alpha+\cdots+c_1(V)\alpha^{r-1}+\alpha^r\in H_T^{2r}(X).
\]
In particular, the Euler class is invertible in equivariant cohomology, as long as we formally invert the class $\alpha$:
\[
e^{-1}(V\otimes\O_\alpha)=\frac{1}{\alpha^r}\sum_{0\leq k\leq \dim(X)}\left(-\frac{c_1(V)}{\alpha}-\cdots-\frac{c_{r-1}(V)}{\alpha^{r-1}}-\frac{c_r(V)}{\alpha^r}\right)^k.
\]
More generally, if $T$ acts trivially on $X$ and $V$ is an equivariant vector bundle such that
\[
X=\mathrm{Tot}(V)^T,
\]
then $e(V)$ is invertible in $H^*_T(X)$, after formally inverting the some of the equivariant parameters.
\end{example}

The final example concerns the standard torus action on projective space, where both the geometry and the action are non-trivial.

\begin{example}\label{ex:projective}
Let $T=(\C^*)^{n+1}$ act on $\P^n$ by
\[
[x_0,\dots,x_n]\mapsto [t_0x_0,\dots,t_nx_n].
\]
Then
\begin{equation}\label{eq:equivproj}
H_T^*(\P^n)=\frac{\Q[H,\alpha_0,\dots,\alpha_n]}{(H-\alpha_0)\cdots(H-\alpha_n)}.
\end{equation}
In usual cohomology, we interpret $H$ as a first Chern class, $H=c_1(\O_{\P^n}(1))$, and that is essentially the same here. Recall that the total space of $\O_{\P^n}(1)$ can be defined globally by
\begin{equation}\label{eq:equivhyperplane}
\mathrm{Tot}(\O_{\P^n}(1))=\frac{(\C^{n+1}\setminus\{0\})\times\C}{\C^*},
\end{equation}
where $\C^*$ acts on the $n+2$ coordinates in the numerator diagonally; the first $n+1$ coordinates $(x_0,\dots,x_n)$ correspond to the homogeneous coordinates on $\P^n$ and the last coordinate $v$ corresponds to the fiber coordinate. The quotient \eqref{eq:equivhyperplane} can be made equivariant in a canonical way: 
\[
[x_0,\dots,x_n,v]\mapsto[t_0x_0,\dots,t_nx_n,v],
\]
and in the $T$-equivariant geometry of $\P^n$, we will always take $\O_{\P^n}(1)$ to denote \eqref{eq:equivhyperplane} with this $T$-action, and we define
\[
H:=c_1(\O_{\P^n}(1))\in H_T^*(\P^n).
\]

Tensoring $\O_{\P^n}(1)$ by $\O_\alpha$ is the same as letting $T$ act on the last coordinate of \eqref{eq:equivhyperplane} by $v\mapsto t^{-\alpha} v$. From the description \eqref{eq:equivhyperplane}, we see that $v=x_i$ descends to a $T$-equivariant section of $\O_{\P^n}(1)\otimes\O_{-\alpha_i}$. In fact, the set of $T$-equivariant sections of $\O_{\P^n}(1)\otimes\O_{-\alpha_i}$ are the scalar multiples of $v=x_i$. Since top Chern classes are often interpreted as the (Poincar\'e dual of) the vanishing of a generic section, we should interpret $H-\alpha_i$ as the vanishing of $x_i$:
\[
H-\alpha_i=[\{x_i=0\}].
\]
Taken a step further, we interpret
\[
\prod_{i\in I}(H-\alpha_i)=[\{x_i=0: \forall i\in I\}],
\]
and the relation in \eqref{eq:equivproj} has the natural interpretation that $\{x_i=0:\forall i\}=\emptyset$.

If $i_j:p_j\rightarrow \P^n$ is the inclusion of the $j$th coordinate point for $0\leq j\leq n$, then the description \eqref{eq:equivhyperplane} implies that $i_j^*\O_{\P^n}(1)=\O_{\alpha_j}$. Thus, by restricting $H$ to fixed points, we obtain
\[
i_j^*H=\alpha_j.
\]
Similarly, restricting the tangent bundle to $p_j$ and using the local coordinates $z_k=x_k/x_j$, we obtain
\[
i_j^*T\P^n=\bigoplus_{k\neq j}\C\left\{\frac{\partial}{\partial z_k} \right\}=\bigoplus_{k\neq j}\O_{\alpha_j-\alpha_k}.
\]

\begin{figure}
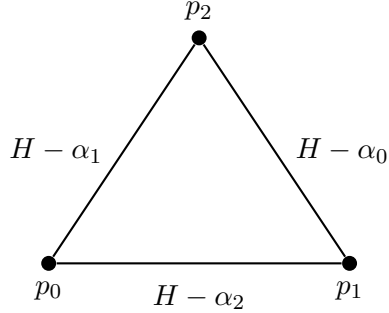

\begin{center}
\tikz{
\node[circle, fill=black, inner sep=2pt, label=below:{$p_0$}] (a) at (0,0) {};
\node[circle, fill=black, inner sep=2pt, label=below:{$p_1$}] (b) at (4,0) {};
\node[circle, fill=black, inner sep=2pt, label=above:{$p_2$}] (c) at (2,3) {};
\draw[thick] (a)--(b) node [midway, label=below:{$H-\alpha_2$}] {};
\draw[thick] (b)--(c) node [midway, label=right:{$H-\alpha_0$}] {};
\draw[thick] (c)--(a) node [midway, label=left:{$H-\alpha_1$}] {};
}
\end{center}
\caption{The $T=(\C^*)^3$-equivariant structure of $\P^2$.}
\end{figure}

\end{example}

\subsection{Localization}

Let $i:F\hookrightarrow X$ be a closed subvariety. Via Poincar\'e duality, there is a well-defined push-forward map $i_*:H^*(F)\rightarrow H^{*+2r}(X)$, where $r$ is the complex codimension of $F$ in $X$. The excess intersection formula tells us that pulling back the image of $i$ by $i^*$ is the same as multiplying by the Euler class of the normal bundle:
\begin{align*}
i^*i_*:H^*(F) &\rightarrow \mathrm{H}^*(F)\\
\phi&\mapsto e(N_{F/X})\phi.
\end{align*}
Clearly, $i^*i_*$ is not an isomorphism, because the Euler class of the normal bundle is not generally invertible. However, if $F$ is the fixed locus of a $T$-action on $X$, then we saw in Example \ref{ex:trivialaction} that the Euler class is invertible after inverting (some of) the equivariant parameters. This is the main motivation for the following result.

\begin{theorem}[Atiyah-Bott \cite{AB}]
Let $X$ be a smooth projective variety with $T=(\C^*)^n$ action such that $F=X^T$ is smooth. Upon inverting the equivariant parameters $\{\alpha_i\}$, $i_*:H_T^*(F)\rightarrow H_T^{*+2r}(X)$ is an isomorphism with inverse
\[
(i_*)^{-1}=\frac{i^*}{e(N_{F/X})}.
\]
\end{theorem}


\begin{remark}
While the fixed locus is required to be smooth, it is almost always reducible, having many connected components. When applying the localization theorem, it is often the case that computations boil down to a clever combinatorial packaging of torus fixed loci.
\end{remark}

\begin{remark}
The localization theorem holds in much greater generality than stated here. For the purposes of these notes, it is important to note that it remains true, for example, if \emph{variety} is replaced with \emph{Deligne-Mumford stack}.
\end{remark}

Recall that the integration map is defined by $\int_X-:=\pi_*(-):H_{(T)}^*(X)\rightarrow H_{(T)}^*(\mathrm{pt})$ where $\pi:X\rightarrow\mathrm{pt}$ can be interpreted with or without a $T$-action on $X$. As an immediate corollary of the localization theorem, we see that
\[
\int_X\phi=\int_Xi_*\left(\frac{i^*\phi}{e(N_{F/X})}\right)=\int_F\frac{i^*\phi}{e(N_{F/X})},
\]
where the last step uses the fact that $F\rightarrow\mathrm{pt}$ factors through $X$. Thus, we obtain a powerful strategy for computing $\int_X\phi$:
\begin{enumerate}
\item[Step 1.] Put a $T$-action on $X$.
\item[Step 2.] Choose a lift of $\phi$ to $H_T^*(X)$.
\item[Step 3.] Compute $\int_F\frac{i^*\phi}{e(N_{F/X})}$.
\end{enumerate}

\begin{remark}
In these notes, the classes we integrate will always be given by Chern classes of vector bundles, and the choice of lift in Step 2 corresponds to the choice of a $T$-action on the corresponding vector bundle.
\end{remark}

\begin{remark}
Since $F$ often has many components, the integral in Step 3 can be a rather large sum. While each summand is typically computable, it can be highly nontrivial to package the sum in a nice way. However, the choice of lift in Step 2 buys us some flexibility. Often by lifting the class cleverly, you can get it to vanish on some of the fixed loci.
\end{remark}

The following examples provide the first applications of this strategy.

\begin{example}
The Euler characteristic of $\P^n$ is the integral of the Euler class of the tangent bundle:
\[
\chi(\P^n)=\int_{\P^n}e(T\P^n).
\]
Using the standard action of $(\C^*)^{n+1}$ on $\P^n$ from Example \ref{ex:projective} (or any other torus action with the same fixed points), we compute
\[
\int_{\P^n}e(T\P^n)=\sum_{j=0}^n\frac{e(i_j^*T\P^n)}{e(N_{p_j/X})}=\sum_{j=0}^n\frac{e(i_j^*T\P^n)}{e(i_j^*T\P^n)}=n+1.
\]
\end{example}

\begin{exercise}
Use localization to compute $\chi(\mathrm{Gr}(k,n))$.
\end{exercise}

\begin{example}
We know that $\int_{\P^2}H^2=1$, so let's try to recover this using localization. If we choose the canonical lift of $H$, as described in Example \ref{ex:projective}, we compute
\[
\int_{\P^2}H^2=\sum_{j=0}^2\frac{i_j^*H^2}{e(i_j^*T\P^2)}=\frac{\alpha_0^2}{(\alpha_0-\alpha_1)(\alpha_0-\alpha_2)}+\frac{\alpha_1^2}{(\alpha_1-\alpha_0)(\alpha_1-\alpha_2)}+\frac{\alpha_2^2}{(\alpha_2-\alpha_0)(\alpha_2-\alpha_1)}=\dots=1,
\]
where the last step required a bit of magical cancellation. If, instead, we lift the class $H\in H^*(\P^2)$ to $H-\alpha_2\in H_T^*(\P^2)$, then the computation would only involve two summands:
\[
\int_{\P^2}H^2=\dots=\frac{(\alpha_0-\alpha_2)^2}{(\alpha_0-\alpha_1)(\alpha_0-\alpha_2)}+\frac{(\alpha_1-\alpha_2)^2}{(\alpha_1-\alpha_0)(\alpha_1-\alpha_2)}=\dots=1,
\]
and the cancellation becomes slightly less magical. Thus, we see how different choices of lifts can change the complexity of the localization computations.
\end{example}

\begin{exercise}
Find a lift of $H^n$ so that only one fixed point contributes to the computation of $\int_{\P^n}H^n=1$.
\end{exercise}

\begin{exercise}\label{ex:localization}
Let $T=(\C^*)^{n+1}$ act on $\M_{0,0}(\P^n,d)$ by post-composing each stable map with the standard action of $T$ on $\P^n$.
\begin{enumerate}
\item[(a)] Prove that the stable map
\begin{align*}
\P^1&\rightarrow\P^n\\
[x_0,x_1]&\mapsto [x_0^d,y_0^d,0,\dots,0]
\end{align*}
is a fixed point of the $T$-action.
\item[(b)] What is the rank of the normal bundle to this fixed point?
\item[(c)] What is the Euler class of the normal bundle to this fixed point?
\end{enumerate}
\end{exercise}

\section{Fan-Jarvis-Ruan-Witten theory of the quintic threefold}

This lecture introduces FJRW theory and the genus-zero FJRW mirror theorem. We do not aim for completeness in our development of FJRW theory; rather, our intent is to introduce the structures that will be most important for the proof of the mirror theorem. For the reader that would like to learn more, we provide several exercises and specific references to places in the literature where the foundations are developed more thoroughly.

\subsection{FJRW theory of the quintic threefold}

FJRW invariants are a special class of intersection numbers on \emph{moduli spaces of $5$-spin curves}. The closed points of the moduli space of $5$-spin curves are defined by
\[
\M_{g,(m_1,\dots,m_n)}^{1/5}=\left\{ (C;q_1,\dots,q_n;L;f)\right\}/\sim
\]
where
\begin{itemize}
\item $(C;q_1,\dots,q_n)$ is a stable, $n$-pointed, genus-$g$ orbifold curve, with orbifold structure only at the marks and nodes;
\item $L$ is an orbifold line bundle on $C$ such that
\begin{itemize}
\item $L$ has multiplicity $m_i$ at $q_i$, i.e. near $q_i$, the total space of $L$ is a quotient by the cyclic isotropy group $\mu_5=\langle\xi=\mathrm{e}^{2\pi\mathrm{i}/5}\rangle$ acting by $\xi\cdot(x,v)=(\xi x,\xi^{m_i} v)$ with $m_i\in\{1,2,3,4\}$;
\item $L$ satisfies the kissing condition at each node, meaning that the multiplicities on each branch satisfy $m+m'=0 \mod 5$ (if $m=m'=0$, then $C$ and $L$ have trivial orbifold structure near that node);
\end{itemize}
\item $f$ is an isomorphism 
\[
f:L^{\otimes 5}\stackrel{\cong}{\longrightarrow}\omega_{C,\log}:=\omega_C\left(\sum_i5[q_i]\right),
\]
where $[q_i]$ is the orbifold divisor of the marked point ($\deg([q_i])=1/5$).
\end{itemize}
An isomorphism of two $5$-spin curves $(C;q_1,\dots,q_n;L;f)\sim(C';q_1',\dots,q_n';L;f)$ consists of a pair of isomorphisms $g:(C;q_1,\dots,q_n)\rightarrow (C',q_1',\dots,q_n')$ and $h:g^*L'\rightarrow L$ such that $h^*f=g^*f'$. We use $\M_{g,n}^{1/5}$ to denote the disjoint union over all possible multiplicity vectors.

\begin{remark}
If this is your first encounter with the moduli spaces of spin curves, then you might naturally be wondering about the necessity of the orbifold structure and the role of the multiplicities. Without orbifold structure at the marked points, a fifth root of $\omega_{C,\log}$ would only exist if $2g-2+n\in 5\Z$, so most of these moduli spaces would be empty. Allowing orbifold structure gives us the flexibility of line bundles with fractional degree and results in a much richer class of moduli spaces with which we can work. More importantly, perhaps, is the observation that, regardless of the orbifold structure at the marked points, the orbifold structure at the nodes is necessary in order to obtain a proper moduli space. For example, if $g=6$ and $n=0$, then any smooth curve has a (non-orbifold) fifth root of $\omega_{C}$, but there does not exist a (non-orbifold) fifth root on any limit point where the curve splits as a $g=1$ curve meeting a $g=5$ curve at a node.
\end{remark}

The moduli space $\M_{g,\vec m}^{1/5}$ is a smooth Deligne-Mumford stack of dimension $3g-3+n$ and support a \emph{virtual funamental class} of (complex) dimension $2n-\sum m_i$:
\[
[\M_{g,\vec m}^{1/5}]^\vir\in H_{4n-2\sum m_i}(\M_{g,\vec m}^{1/5}).
\]
The construction of the virtual class is highly non-trivial, but has been carried out in several equivalent settings \cite{PV,C,FJR,CLL}. FJRW invariants are defined by
\[
\langle\phi_{m_1}\psi^{a_1}\cdots\phi_{m_n}\psi^{a_n}\rangle_{g,n}^{\mathrm{FJRW}}:=\int_{[\M_{g,\vec m}^{1/5}]^\vir}\psi_1^{a_1}\cdots\psi_n^{a_n}\in\Q,
\]
where $\psi_i$ is the cotangent line class at the $i$th marked point on the coarse underlying curve $|C|$.

The general construction of the virtual class is far beyond the scope of these lectures; however, the construction simplifies a great deal in genus zero. In particular, the genus-zero virtual class can be defined by
\[
\left[\M_{0,\vec m}^{1/5}\right]^{\vir}:=\left[\M_{0,\vec m}^{1/5}\right]\cap e\left( (R^1\pi_*\L^{\oplus 5})^\vee \right),
\]
where $\pi:\mathcal{C}\rightarrow\M_{0,\vec m}^{1/5}$ is the universal curve and $\L$ is the universal line bundle over $\mathcal{C}$.

To gain a slightly better appreciation of the genus-zero virtual class, it is instructive to verify that (a) $R^1\pi_*\L$ is, in fact, a vector bundle, so its Euler class is well-defined, and (b) the virtual dimension is $2n-\sum m_i$. In order to make these verifications, let us first say a few things about orbifold line bundles. Suppose $C$ is a smooth curve. Locally at $q_i$, sections of $L$ correspond to sections in the orbifold coordinates that are equivariant with respect to the isotropy group action. In other words, local sections are of the form $v(x)=x^{m_i}F(x^5)$ for some polynomial $F$. The variable $x^5$ is a local coordinate on the coarse underlying curve. Thus, we see that the orbifold multiplicity $m_i$ imposes a vanishing of degree $m_i/5$ at each orbifold point. We define the round-down of $L$ by
\[
|L|:=L\otimes\O\left(-\sum m_i[q_i]\right).
\]
It follows from the discussion above that $|L|$ is pulled back from the coarse underlying curve $|C|$, and, since rounding up/down by $m_i$ identifies sections, we have $H^i(C,L)=H^i(|C|,|L|)$ for $i=0,1$. In particular, these observations allow us to import the classical Riemann-Roch theorem to this orbifold setting:
\[
H^0(C,L)-H^1(C,L)=\deg(L)+1-g-\sum\frac{m_i}{5}.
\]

Using these observations, the reader is encouraged to work out the following verifications.

\begin{exercise}\label{ex:concave}
Prove that $R^1\pi_*\L$ is a vector bundle by showing that $H^0(C,L)=0$ for any $(C,q_1,\dots,q_n,L,f)\in \M_{0,\vec m}^{1/5}$. (Hint: compute $\deg(|L|)$ on each irreducible component.) In our definition, we do not allow $m_i=0$. Would your proof still work if one of the $m_i$ is allowed to be zero?
\end{exercise}

\begin{exercise}
Use Riemann-Roch to prove that the genus-zero virtual class lies in complex dimension $2n-\sum m_i$.
\end{exercise}

There are a few additional properties concerning the FJRW invariants that are worth mentioning. First of all, the multiplicity-one insertion $\phi_1$ plays a special role in this story. In particular, there is a forgetful map
\[
\mathrm{f}:\M_{g,(\vec m,1)}^{1/5}\rightarrow\M_{g,\vec m}^{1/5}.
\]
Roughly speaking, this map rounds down $L$ at the last marked point, then forgets that point and stabilizes, if necessary. This only works with $\phi_1$ insertions because rounding down a $5$-spin bundle at a $\phi_1$ insertion produces a $5$-spin bundle on the curve without that marked point:
\[
L^{\otimes 5}\cong\omega_{C}\left(\sum_{i=1}^{n+1}5[q_i]\right)\Longrightarrow \left( L\otimes\O\left(-[q_{n+1}]\right)\right)^{\otimes 5}\cong \omega_{C}\left(\sum_{i=1}^{n}5[q_i]\right).
\]

\begin{exercise}
Prove that $\mathrm{f}^*\left[\M_{0,\vec m}^{1/5}\right]^\vir=\left[\M_{0,(\vec m,1)}^{1/5}\right]^\vir$. Along with the comparison lemma for $\psi$-classes (\cite{MS}, Lemma 25.2.3), deduce the string and dilaton equations:
\[
\left\langle\phi_{m_1}\psi^{a_1}\cdots\phi_{m_n}\psi^{a_n}\cdot \phi_1\right\rangle_{0,n+1}^\mathrm{FJRW}=\sum_{i=1}^n\left\langle\phi_{m_1}\psi^{a_1}\cdots\phi_{m_i}\psi^{a_i-1}\cdots\phi_{m_n}\psi^{a_n}\right\rangle_{0,n}^\mathrm{FJRW}
\]
and
\[
\left\langle\phi_{m_1}\psi^{a_1}\cdots\phi_{m_n}\psi^{a_n}\cdot\phi_1\psi\right\rangle_{0,n+1}^\mathrm{FJRW}=(n-2)\left\langle\phi_{m_1}\psi^{a_1}\cdots\phi_{m_n}\psi^{a_n}\right\rangle_{0,n}^\mathrm{FJRW}.
\]
\end{exercise}

A final property of the genus-zero FJRW invariants that is important for our techniques concerns how the virtual class restricts to the boundary. Since we do not allow multiplicity zero at the marked points, then it would be ideal if the virtual class vanished along nodes where the multiplicities were zero. Essentially, this implies that the genus-zero virtual class restricts to the boundary in a way that is compatible with the gluing map 
\[
\mathrm{g}:\M_{0,(\vec m_I)+1}^{1/5}\times\M_{0,(\vec m_J)+1}^{1/5}\rightarrow\M_{0,\vec m}^{1/5}.
\]
This is indeed the case, as the next exercise verifies.

\begin{exercise}
Let $D\subset\M_{0,\vec m}^{1/5}$ denote a divisor where the node has multiplicity $m=m'=0$. Prove that $\left[\M_{0,\vec m}^{1/5}\right]^\vir\cap D=0$.
\end{exercise}

\begin{remark}
This exercise is a bit harder than the previous ones. The trick is to use the normalization sequence and show that $R^1\pi_*\L$ contains a sub-line bundle that has a trivial first Chern class (c.f. \cite{CR}, Lemma 4.1.1).
\end{remark}

\subsection{The genus-zero FJRW mirror theorem}

In order to state the mirror theorem, it is useful to interpret FJRW invariants as multi-linear functions on the vector space $V=\C\{\phi_1,\dots,\phi_4\}$, where we have one basis element for each possible multiplicity. The space $V$ has a non-degenerate pairing defined by $(\phi_i,\phi_j)=\frac{1}{5}\delta_{i+j=5}$, and we define the dual basis of $V$ by $\phi^i=5\phi_{i-1}$.

The two main players in the mirror theorem are typically called the $J$-function and the $I$-function. The $J$-function is the following formal generating series of genus-zero FJRW invariants:
\[
J(\t,z):=z\phi_1+\mathbf{t}(-z)+\sum_{n\geq 2 \atop 1\leq m\leq 4} \frac{\phi^m}{n!}\left\langle \t(\psi)^n\frac{\phi_m}{z-\psi}\right\rangle_{0,n+1}^\mathrm{FJRW},
\]
where
\begin{itemize}
\item $\t(z)=\sum t_k^m\phi_mz^k\in V[[z]]$ and $t_k^m$ are considered as formal variables,
\item $\t(\psi)^n=\t(\psi_1)\cdots\t(\psi_n)$ is expanded linearly, and
\item $(z-\psi)^{-1}$ is expanded as a geometric series in $\psi$.
\end{itemize}
The $I$-function is the explicit formal series defined by
\[
I(q,z):=z\sum_{d\geq 0}\frac{q^d}{z^dd!}\left(\prod_{0\leq k<\frac{d+1}{d} \atop \langle k \rangle = \langle \frac{d+1}{5}\rangle}(kz)^5 \right)\phi_{d+1}.
\]
By counting the number of terms in the product, one sees that $I(q,z)$ can be written in the form
\[
I(q,z)=I_0(q)z\phi_1+I_1(q)\phi_2+I_2(q)\frac{\phi_3}{z}+I_3(q)\frac{\phi_4}{z^2},
\]
and we write
\[
I_+(q,z):=I_0(q)z\phi_1+I_1(q)\phi_2,
\]
for the part of $I(q,z)$ with non-negative powers of $z$. The FJRW mirror theorem can be stated as follows.

\begin{theorem}[Chiodo-Ruan \cite{CR}]\label{FJRWMT}
Set \[\tau=I_+(q,-z)+z\phi_1.\] Then \[J(\tau,z)=I(q,z).\]
\end{theorem}

\begin{remark}
The change of variables ensures that $J(\tau,z)=I(q,z)$ mod $z^{-1}$. The content of the theorem is that the coefficients of negative powers of $z$ also agree.
\end{remark}

Upon further inspection, the reader might notice that the formulation in Theorem \ref{FJRWMT} is not quite the same as that in \cite{CR}. However, it is not too hard to prove that they are equivalent.
\begin{exercise}
Prove that Theorem \ref{FJRWMT} is equivalent to the more standard formulation:
\[
J\left(\frac{I_1(q)}{I_0(q)}\phi_2,z\right)=\frac{I(q,z)}{I_0(q)}.
\]
(Hint: use the dilaton equation.)
\end{exercise}

The formulation in the previous exercise immediately determines the genus-zero FJRW invariants with any number of $\phi_2$ insertions and one additional insertion, possibly with a $\psi$-class. However, the next exercise shows that, actually, all genus-zero FJRW invariants are determined by the mirror theorem.

\begin{exercise}
Prove that the formulation in the previous exercise determines \underline{all} genus-zero FJRW invariants. (Hint: use the string and dilaton equations along with the formula for the virtual dimension.)
\end{exercise}

\section{Proof of the Fan-Jarvis-Ruan-Witten mirror theorem}\label{lecture:proofFJRW}

In this lecture, we present a proof of the FJRW mirror theorem, as stated in Theorem \ref{FJRWMT}, by showing that it is a consequence of certain localization relations on auxiliary moduli spaces that have a natural $T=\C^*$ action.

\subsection{Auxiliary moduli spaces}

Consider moduli spaces
\[
\G\M_{0,\vec m,d}^{1/5}=\left\{ (C;q_1,\dots,q_n;L;D;f;g)\right\}/\sim
\]
where
\begin{itemize}
\item $(C;q_1,\dots,q_n)$ is a pre-stable, $n$-pointed, genus-zero orbifold curve, with orbifold structure only at the marks and nodes;
\item $L$ is an orbifold line bundle on $C$ such that
\begin{itemize}
\item $L$ has multiplicity $m_i$ at $q_i$, with $m_i\in\{1,2,3,4\}$;
\item $L$ satisfies the kissing condition at each node;
\end{itemize}
\item $D$ is an effective degree-$d$ divisor on $C$, supported away from the marks and nodes,
\item $f$ is an isomorphism 
\[
f:L^{\otimes 5}\stackrel{\cong}{\longrightarrow}\omega_{C,\log}\otimes\O(-D);
\]
\item $g:C\rightarrow \P^1$ is a degree-$1$ map; and
\item all of this data satisfies the stability condition: $\omega_{C,\log}\otimes\O(\epsilon D)\otimes g^*\O_{\P^1}(3)$ is ample for all $\epsilon>0$.
\end{itemize}
Isomorphisms are required to preserve the divisor $D$ and commute with the map $g$.

These conditions are a lot to swallow, so let us parse the main ideas. If $d=0$, then $\G\M_{0,\vec m,0}^{1/5}$ parametrizes usual $5$-spin curves along with an additional degree-one map to $\P^1$. This additional map parametrizes one irreducible component $C_0\subseteq C$, and the stability condition simply says that $C_0$ is always stable, regardless of how many special points it has, while every other component is stable if and only if it has at least three special points. If $d\neq 0$, then any non-parametrized component is stable if and only if it satisfies one of the following two conditions (see Figure \ref{fig:stable}):
\begin{enumerate}
\item it has at least three special points (marked points or nodes), or
\item it has two special points and intersects $\mathrm{Supp}(D)$ nontrivially.
\end{enumerate}

\begin{figure}
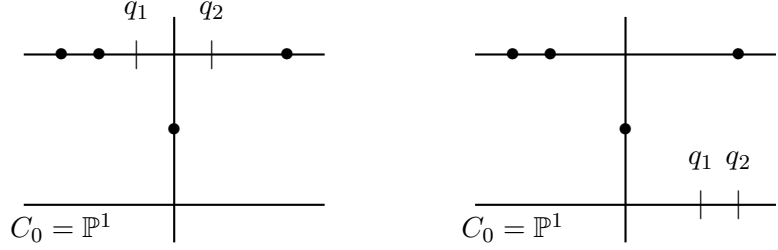

\begin{center}
\tikz{
\draw[thick] (0,0) -- (4,0);
\draw[thick] (2,-.5) -- (2,2.5);
\draw[thick] (0,2) -- (4,2);
\node [] at (.5,-.3) {$C_0=\P^1$};
\node [] at (2,1) {$\bullet$};
\node [] at (.5,2) {$\bullet$};
\node [] at (1,2) {$\bullet$};
\node [label=above:{$q_1$}] at (1.5,2) {$|$};
\node [label=above:{$q_2$}] at (2.5,2) {$|$};
\node [] at (3.5,2) {$\bullet$};

\draw[thick] (6,0) -- (10,0);
\draw[thick] (8,-.5) -- (8,2.5);
\draw[thick] (6,2) -- (10,2);
\node [] at (6.5,-.3) {$C_0=\P^1$};
\node [] at (8,1) {$\bullet$};
\node [] at (6.5,2) {$\bullet$};
\node [] at (7,2) {$\bullet$};
\node [] at (9.5,2) {$\bullet$};
\node [label=above:{$q_1$}] at (9,0) {$|$};
\node [label=above:{$q_2$}] at (9.5,0) {$|$};
}
\end{center}
\caption{The left-hand image depicts a stable element of $\G\M_{0,2,4}^{1/5}$, where the dashes represent marked points and the dots represent points in the support of $D$. The right-hand image depicts an unstable object, because no amount of degree in $D$ will stabilize the top component.}\label{fig:stable}
\end{figure}

As in the case of usual $5$-spin curves, $\G\M_{0,\vec m,d}^{1/5}$ has a virtual class, defined by
\[
\left[\G\M_{0,\vec m,d}^{1/5}\right]^\vir:=\left[\G\M_{0,\vec m,d}^{1/5}\right]\cap e\left( (R^1\pi_*\L^{\oplus 5})^\vee \right),
\]
and this virtual class vanishes on divisors where the node has multiplicity zero. Unlike the usual case of $5$-spin curves, this moduli space has a torus action. More specifically, define an action of $T=\C^*$ on $\G\M_{0,\vec m,d}^{1/5}$ by post-composing the map $g:C\rightarrow\P^1$ with the action $t\cdot[x_0,x_1]=[tx_0,x_1]$ on $\P^1$. This is equivalent to acting directly on $|C_0|\stackrel{g}{=}\P^1$. Let $z$ denote the $T$-equivariant parameter, so that
\[
c_1(T\P^1|_0)=z=-c_1(T\P^1|_\infty),
\]
where $0=[1,0]$ and $\infty=[0,1]$. A point of $\G\M_{0,\vec m,d}^{1/5}$ is fixed by the $T$-action if and only if
\[
C_0\cap\left(\left\{\overline{C\setminus C_0}\right\}\cup\{q_1,\dots,q_n\}\cup \mathrm{Supp(D)}\right)\subseteq\{0,\infty\}.
\]
See Figure \ref{fig:fixed}.

\subsection{Unification of the $J$- and $I$-functions}

Let $i:\Gamma_0\hookrightarrow \G\M_{0,n+1,d}^{1/5}$ be the $T$-fixed locus where $q_{n+1}=\infty\in C_0$, and define 
\[
\J(\t,q,z):=-z^2\sum_{{n\geq 0 \atop d\geq 0}\atop 1\leq m\leq 4 }\frac{q^d}{n!}\int_{\Gamma_0}\frac{i^*\left[\G\M_{0,n+1,d}^{1/5}\cap\t(\psi)^n\cap \phi_m\right]^\vir}{e(N_{\Gamma_0})}\phi^m.
\]

\begin{figure}
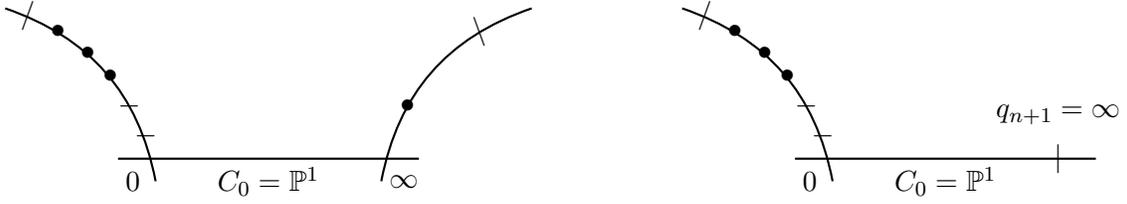

\begin{center}
\tikz{
\draw[thick] (0,0) -- (4,0);
\draw[thick] (.5,-.3) edge[-,bend right] (-1.5,2);
\draw[thick] (3.5,-.3) edge[-,bend left] (5.5,2);
\node [] at (2,-.3) {$C_0=\P^1$};
\node [] at (.2,-.3) {$0$};
\node [] at (3.8,-.3) {$\infty$};

\node [] at (.37,.3) {$-$};
\node [] at (.15,.7) {$-$};
\node [] at (-.1,1.1) {$\bullet$};
\node [] at (-.4,1.4) {$\bullet$};
\node [] at (-.8,1.7) {$\bullet$};
\node [] at (-1.2,1.9) {$/$};
\node [] at (3.85,.7) {$\bullet$};
\node [] at (4.8,1.7) {$\backslash$};

\draw[thick] (9,0) -- (13,0);
\draw[thick] (9.5,-.3) edge[-,bend right] (7.5,2);
\node [] at (11,-.3) {$C_0=\P^1$};
\node [] at (9.2,-.3) {$0$};

\node [] at (9.37,.3) {$-$};
\node [] at (9.15,.7) {$-$};
\node [] at (8.9,1.1) {$\bullet$};
\node [] at (8.6,1.4) {$\bullet$};
\node [] at (8.2,1.7) {$\bullet$};
\node [] at (7.8,1.9) {$/$};
\node [label=above:{$q_{n+1}=\infty$}] at (12.5,0) {$|$};

}
\end{center}
\caption{The left-hand image depicts a generic $T$-fixed point of $\G\M_{0,n,d}^{1/5}$, where the marked points and support of $D$ are distributed on the two branches attached to $C_0$ at $0$ and $\infty$. The right-hand image depicts the distinguished fixed locus $\Gamma_0\subset\G\M_{0,n+1,d}^{1/5}$.}\label{fig:fixed}
\end{figure}

\noindent The formal series $\J(\t,q,z)$ unifies the $J$- and $I$-functions in the following sense.

\begin{proposition}\label{prop:unify}
Setting $q=0$, 
\[
\J(\t,0,z)=J(\t,z).
\]
Setting $\t=0$, 
\[
\J(0,q,z)=I(q,z).
\]
\end{proposition}

\begin{proof}
First, consider the case $q=0$. The first three images in Figure \ref{fig:FJRWJ} depict the three types of fixed loci and how they contribute to $J(\t,z)$. We explain each of these computations in turn. 

If $n=0$, then $\deg(\omega_{C,\log})=-1$, and the moduli space is non-empty only if $m=4$, in which case $\Gamma_0$ is a point (with an order-$5$ automorphism). One checks that $H^1(C,L)=0$, so the virtual class restricts to the usual fundamental class. The normal bundle has a single factor of $-z=c_1(TC_0|_\infty)$ corresponding to moving $q_1$ away from $\infty$. Dividing by the automorphism group cancels with the $5$ in the dual basis: $\phi^4=5\phi_1$. Thus, the coefficient works out to be $z\phi_1$.

If $n=1$, then $\Gamma_0$ is still a point (with an order-$5$ automorphism). On this fixed locus, $\psi_1=c_1(TC_0^\vee|_0)=-z$. The virtual class is the usual fundamental class again, and the normal bundle has two contributions $z$ and $-z$ corresponding to moving the two marked points away from $0$ and $\infty$. Putting everything together, the coefficient in this case is $\t(-z)$.

For $n\geq 2$, $\Gamma_0\cong\M_{0,n+1}^{1/5}$. Using the normalization sequence
\[
0\rightarrow\O_C\rightarrow\O_{C_0}\oplus\O_{\overline{C\setminus C_0}}\rightarrow\O_q\rightarrow 0,
\] 
one checks that the virtual class on $\G\M_{0,n+1}^{1/5}$ restricts to the virtual class on $\M_{0,n+1}^{1/5}$. The normal bundle has two factors $z$ and $-z$ corresponding to moving the node away from $0$ and moving $q_{n+1}$ away from $\infty$, respectively, and one factor of $z-\psi$ corresponding to smoothing the node at $0$. Thus, these coefficients are equal to the $n\geq 2$ terms of $J(\t,z)$.

\begin{figure}
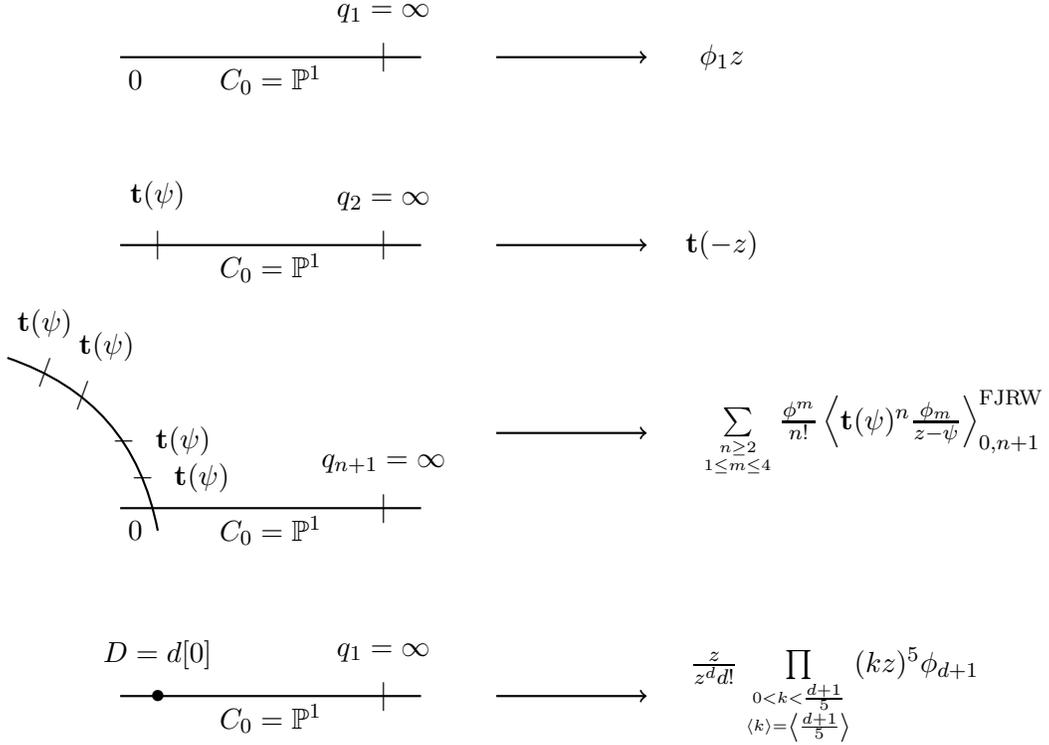

\begin{center}
\tikz{
\draw[thick] (0,0) -- (4,0);
\node [] at (2,-.3) {$C_0=\P^1$};
\node [] at (.2,-.3) {$0$};
\node [label=above:{$q_{1}=\infty$}] at (3.5,0) {$|$};
\draw[thick] (5,0) edge[->] (7,0);
\node[] at (8,0) {$\phi_1z$};

\draw[thick] (0,-2.5) -- (4,-2.5);
\node [] at (2,-2.8) {$C_0=\P^1$};
\node [label=above:{$q_{2}=\infty$}] at (3.5,-2.5) {$|$};
\node [label=above:{$\t(\psi)$}] at (.5,-2.5) {$|$};
\draw[thick] (5,-2.5) edge[->] (7,-2.5);
\node[] at (8,-2.5) {$\t(-z)$};

\draw[thick] (0,-6) -- (4,-6);
\draw[thick] (.5,-6.3) edge[-,bend right] (-1.5,-4);
\node [] at (2,-6.3) {$C_0=\P^1$};
\node [] at (.2,-6.3) {$0$};
\node [label=right:{$\t(\psi)$}] at (.3,-5.6) {$-$};
\node [label=right:{$\t(\psi)$}] at (.05,-5.1) {$-$};
\node [label=above:{$\;\;\;\;\;\;\t(\psi)$}] at (-.5,-4.5) {$/$};
\node [label=above:{$\t(\psi)$}] at (-1,-4.2) {$/$};
\node [label=above:{$q_{n+1}=\infty$}] at (3.5,-6) {$|$};
\draw[thick] (5,-5) edge[->] (7,-5);
\node[] at (10,-5) {$\sum\limits_{n\geq 2 \atop 1\leq m\leq 4} \frac{\phi^m}{n!}\left\langle\t(\psi)^n\frac{\phi_m}{z-\psi} \right\rangle_{0,n+1}^\mathrm{FJRW}$};

\draw[thick] (0,-8.5) -- (4,-8.5);
\node [] at (2,-8.8) {$C_0=\P^1$};
\node [label=above:{$q_{1}=\infty$}] at (3.5,-8.5) {$|$};
\node [label=above:{$D=d[0]$}] at (.5,-8.5) {$\bullet$};
\draw[thick] (5,-8.5) edge[->] (7,-8.5);
\node[] at (9.5,-8.5) {$\frac{z}{z^dd!}\prod\limits_{0<k<\frac{d+1}{5}\atop \langle k\rangle=\left\langle\frac{d+1}{5}\right\rangle} (kz)^5\phi_{d+1}$};
}
\end{center}
\caption{Contributions to $\J(\t,q,z)$ when $q=0$ (first three) and $\t=0$ (last one).}\label{fig:FJRWJ}
\end{figure}

Next, consider the case $\t=0$. Then $\Gamma_0$ is a point (with an order-$5$ automorphism) where $C=C_0$ and $D$ is entirely supported at $0\in C_0$ (see Figure \ref{fig:FJRWJ}). Since 
\[
L^{\otimes 5}=\omega_{C,\log}\otimes\O(-d)\cong\O(-d-1),
\]
we see that $L\cong \O\left(-\frac{d+1}{5}\right)$ and $m=-(d+1)$ mod $5$. With respect to the open cover consisting of the complement of $0$ and the complement of $\infty$, a basis of \v{C}ech sections is given by
\[
H^1(C,L)=\C\left\{x_0^{-k}x_1^{k-\frac{d+1}{5}}:0<k<\frac{d+1}{5},\; \langle k\rangle=\left\langle\frac{d+1}{5}\right\rangle\right\}.
\]
Since the $T$-action has weight $-kz$ on each basis element, we see that the virtual class restricts as
\[
e\left(H^1(C,L^{\oplus 5})^\vee \right)=\prod_{0<k<\frac{d+1}{5}\atop \langle k\rangle=\left\langle\frac{d+1}{5}\right\rangle} (kz)^5.
\]
The normal bundle contributes $-z$ corresponding to moving the marked point away from $\infty$ and $d!z^d$ corresponding to moving the support of $D$ away from $0$. Combining these contributions recovers the $I$-function.
\end{proof}

\subsection{Proof of the mirror theorem}

We now prove the FJRW mirror theorem.

\begin{proof}[Proof of Theorem \ref{FJRWMT}]
It suffices to prove
\begin{equation}\label{eq:FJRWMT}
\J(t,q,z)=\J(t+\tau,0,z)
\end{equation}
where 
\[
t=\sum_{1\leq m\leq 4}t^m\phi_m\;\;\;\;\text{ and }\;\;\;\; \tau=I_+(q,-z)+\phi_1z.
\]
If \eqref{eq:FJRWMT} holds, then setting $t=0$ recovers Theorem \ref{FJRWMT} via Proposition \ref{prop:unify}. Notice that \eqref{eq:FJRWMT} is true modulo $q$ and it is true modulo $z^{-1}$. Our strategy is to show that both sides of \eqref{eq:FJRWMT} satisfy the same recursion, which determines the coefficients recursively in $(d,n)$ (lexicographic ordering).

Begin with the left-hand side, and consider the equivariant integrals
\begin{equation}\label{eq:graphint}
\int_{\left[\G\M_{0,(\vec m,m),d}^{1/5}\right]^\vir}\mathrm{ev}_D^*(0)\cdot\mathrm{ev}_{n+1}^*(\infty)\in H_T^*(\mathrm{pt})=\Q[z]
\end{equation}
where $\mathrm{ev}_D^*(0)$ denotes the equivariant class where $g^{-1}(0)\cap \mathrm{Supp}(D)\neq\emptyset$ and $\mathrm{ev}_{n+1}^*(\infty)$ denotes the equivariant class where $q_{n+1}\in g^{-1}(\infty)$. By definition, \eqref{eq:graphint} is polynomial in $z$. Therefore, if we invert $z$ and compute the integral by localization, we obtain relations among the coefficients of negative powers of $z$.

To exhibit how these relations can be useful, let us suppose that we are interested in computing the $q^dt^{m_1}\cdots t^{m_n}\phi^m$-coefficient of $\J(t,q,z)$. Computing the integral \eqref{eq:graphint} by localization, the fixed locus $\Gamma_0$ contributes $d$ times the coefficient we are interested in (there is a $dz$ from the restriction of $\mathrm{ev}_D^*(0)$ and a $-z$ from the restriction of $\mathrm{ev}_{n+1}^*(\infty)$). Moreover, the contribution of every other fixed locus (see Figure \ref{fig:fixed}) is determined recursively because the components over $0$ and $\infty$ either have degree $<d$ (and any number of special points), or degree $=d$ but with fewer special points. Thus, with the lexicographic ordering on $(d,n)$, this gives an effective recursion determining $\J(t,q,z)$ from $q=0$ and the coefficients of the non-negative powers of $z$.

Inserting these relations as coefficients of the appropriate generating series, we can package them efficiently as the following statement:
\begin{equation}\label{eq:relations}
\left(\frac{\partial}{\partial q}\J(t,q,z),\frac{\partial}{\partial t^m}\J(t,q,-z) \right)[z^{-k}]=0 \;\;\;\;\forall k>0 \text{ and } 1\leq m\leq 4.
\end{equation}

Now consider the right-hand side, where we investigate the equivariant integrals
\[
\int_{\left[\G\M_{0,(\vec m,m)+1,0}^{1/5}\right]^\vir}\left(\mathrm{ev}_{n+2}^*(0)\frac{\partial\tau}{\partial q}\right)\cdot\mathrm{ev}_{n+1}^*(\infty).
\]
Localizing exactly as before and packaging in the appropriate generating series, we find that
\begin{equation}\label{eq:relations2}
\left(\frac{\partial}{\partial q}\J(t+\tau,0,z),\frac{\partial}{\partial t^m}\J(t+\tau,0,-z) \right)[z^{-k}]=0 \;\;\;\;\forall k>0 \text{ and } 1\leq m\leq 4,
\end{equation}
and the relations \eqref{eq:relations2} determine the $z^{-k}$ coefficients of $\J(t+\tau,0,z)$ from $q=0$ and the coefficients of the non-negative powers of $z$.

Since the recursions \eqref{eq:relations} and \eqref{eq:relations2} are identical, and they determine the series from $q=0$ and the coefficients of the non-negative powers of $z$, then the two sides of \eqref{eq:FJRWMT} are equal.
\end{proof}

\section{Gromov-Witten theory of the quintic threefold}

In this lecture, we introduce GW invariants of the quintic threefold and state the genus-zero GW mirror theorem. We also apply localization to provide an explicit algorithm for computing any genus-zero GW invariant of the quintic threefold.

\subsection{GW theory of the quintic threefold}

GW invariants are a special class of intersection numbers on \emph{moduli spaces of stable maps}. The closed points of the moduli space of stable maps are defined by
\[
\M_{g,n}(Q,d)=\left\{(C;q_1,\dots,q_n;f) \right\}/\sim
\]
where
\begin{itemize}
\item $(C;q_1,\dots,q_n)$ is a pre-stable, $n$-pointed, genus-$g$ curve;
\item $f:C\rightarrow Q$ is a map of degree $d\in H_2(\P^4)=\Z$; and
\item this data satisfies the stability condition: $\omega_{C,\log}\otimes f^*\O_{\P^4}(3)$ is ample.
\end{itemize}
An isomorphism of two points is an isomorphism of pointed curves that commutes with the maps to $Q$. The stability condition is equivalent to only allowing points with finitely-many automorphisms.

The moduli spaces $\M_{g,n}(Q,d)$ are not smooth, even as Deligne-Mumford stacks, and they can have many components of higher-than-expected dimension. It is a highly non-trivial but fundamental fact that $\M_{g,n}(Q,d)$ supports a virtual fundamental class
\[
\left[\M_{g,n}(Q,d)\right]^\vir\in H_{2n}\left(\M_{g,n}(Q,d)\right),
\]
which was constructed in \cite{BF}. For $\varphi_1,\dots,\varphi_n\in H^*(Q)$, GW invariants are defined by
\[
\langle\varphi_{1}\psi^{a_1}\cdots\varphi_{n}\psi^{a_n}\rangle_{g,n,d}^\mathrm{GW}:=\int_{[\M_{g,n}(Q,d)]^\vir}\mathrm{ev}_1^*(\varphi_1)\psi_1^{a_1}\cdots\mathrm{ev}_n^*(\varphi_n)\psi_n^{a_n}\in\Q,
\]
where $\mathrm{ev}_i:\M_{g,n}(Q,d)\rightarrow Q$ evaluates the map $f$ at the $i$th marked point, and $\psi_i$ is the cotangent line class at the $i$th marked point on the pre-stable curve $C$.

In order to state the genus-zero mirror theorem, we require the GW $J$- and $I$-functions. Let $V=H^*(Q)=\Q\{\phi_i\}$, where $\{\phi_i\}$ is a basis of cohomology and let $\{\phi^i\}$ denote the dual basis under the Poincar\'e pairing. Then for $\t(z)\in V[[z]]$, the $J$-function is defined by
\[
J(\t,z):=z+\t(-z)+\sum_{n,d,i}\frac{q^d}{n!}\left\langle \t(\psi)^n\frac{\phi_i}{z-\psi}\right\rangle_{0,n+1,d}^\mathrm{GW}\phi^i,
\]
and the $I$-function is defined by
\[
I(q,z):=z\sum_{d\geq 0}q^d\frac{\prod_{k=1}^{5d}(5H+kz)}{\prod_{k=1}^d(H+kz)^5},
\]
where $H$ is the restriction of the hyperplane class from $H^*(\P^4)$. The $I$-function should be expanded as a polynomial in $H\in H^*(Q)$. Similarly to the FJRW case, we can write
\[
I(q,z)=I_0(q)z+I_1(q)H+I_2(q)\frac{H^2}{z}+I_3(q)\frac{H^3}{z^3},
\]
and we define
\[
I_+(q,z):=I_0(q)z+I_1(q)H.
\]

The mirror theorem can be stated in the following way.

\begin{theorem}[Givental \cite{G}]\label{thm:GWMT}
Set \[\tau=I_+(q,-z)+z\phi_1.\] Then \[J(\tau,z)=I(q,z).\]
\end{theorem}

Exactly as in the FJRW case, the dilaton equation can be used to prove that Theorem \ref{thm:GWMT} is equivalent to the more standard formulation of the mirror theorem:
\[
J\left(\frac{I_1(q)}{I_0(q)}\phi_2,z\right)=\frac{I(q,z)}{I_0(q)},
\]
and the string and dilaton equation, along with the formula for the virtual dimension, show that the mirror theorem determines all genus-zero FJRW invariants of the quintic.

\subsection{Computing genus-zero GW invariants}

Similarly to the FJRW setting, the general construction of the virtual class is far beyond the scope of these lectures. However, the situation greatly simplifies in genus zero. This is due to the fact that  genus-zero stable maps to the quintic embed into $\M_{0,n}(\P^4,d)$, which is a smooth Deligne-Mumford stack, and the push-forward of the virtual class can be described concretely as an Euler class:
\[
i_*[\M_{0,n}(Q,d)]^\vir=[\M_{0,n}(\P^4,d)]\cap e(R^0\pi_*f^*\O_{\P^4}(5)),
\]
where $\pi:\mathcal{C}\rightarrow\M_{0,n}(\P^4,d)$ is the universal curve and $f:\mathcal{C}\rightarrow\P^4$ is the universal map. Therefore, if $\varphi_1,\dots,\varphi_n\in H^*(\P^4)$, the projection formula implies that
\begin{equation}\label{eq:gwinvariants}
\langle\varphi_1\psi^{a_1}\cdots\varphi_n\psi^{a_n}\rangle_{0,n,d}^\mathrm{GW}=\int_{[\M_{0,n}(\P^4,d)]}\left(\mathrm{ev}_1^*(\varphi_1)\psi_1^{a_1}\cdots\mathrm{ev}_n^*(\varphi_n)\psi^{a_n}\right)e(R^0\pi_*f^*\O_{\P^4}(5)),
\end{equation}
where we implicitly restrict $\varphi_i$ to $H^*(Q)$ in the left-hand side. Since $\M_{0,n}(\P^4,d)$ has an action of the torus $T=(\C^*)^5$, the invariants \eqref{eq:gwinvariants} can be computed by localization, and we now provide a brief overview of how this works out.

Each component of the $T$-fixed locus in $\M_{0,n}(\P^4,d)$ can be indexed by a decorated tree. Such a tree $\Gamma$ has vertices $V(\Gamma)$ and edges $E(\Gamma)$. Each vertex $v$ is labeled with an index $j_v\in\{0,\dots,4\}$ and has legs indexed by $I_v\subset\{1,\dots,n\}$, and corresponds to a connected component of $f^{-1}(p_{j_v})$ containing the fixed points $\{q_i:i\in I_v\}$. Each edge $e$ is labeled with a positive integer $d_e$ and corresponds to a connected multiple cover of the corresponding $T$-invariant line in $\P^4$. For example, the fixed locus in Exercise \ref{ex:localization} is indexed by the graph with two vertices labeled $0$ and $1$ connected by an edge labeled $d$. Figure \ref{fig:localization} gives a more general example of a localization graph and its corresponding fixed locus.

\begin{figure}
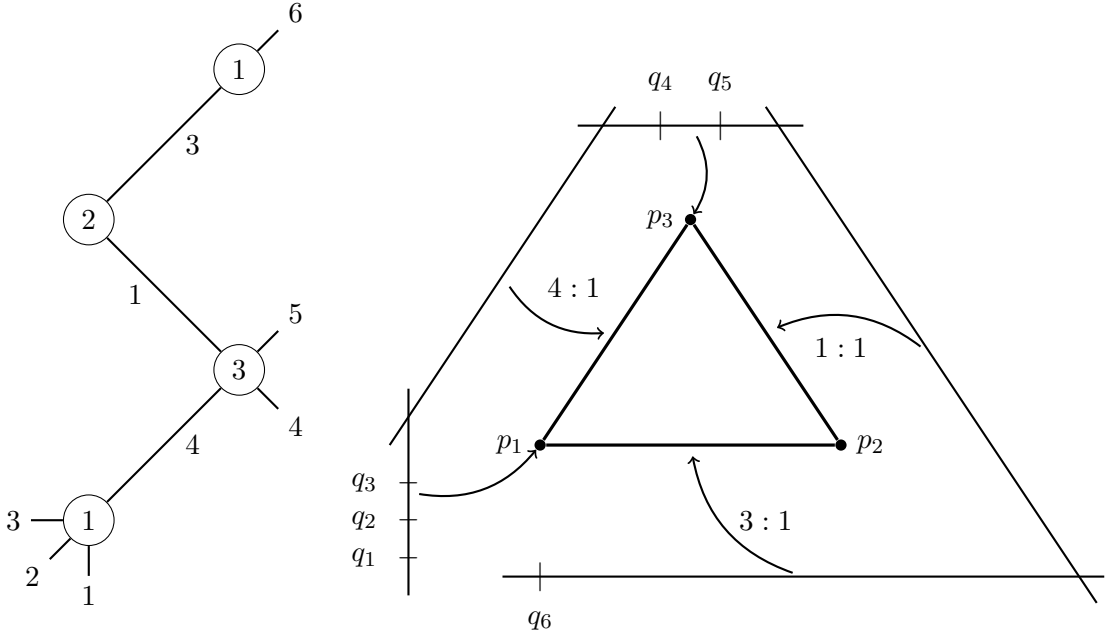

\begin{center}
\tikz{
\node[circle, draw] (a) at (2,0) {1};
\node[circle, draw] (b) at (4,2) {3};
\node[circle, draw] (c) at (2,4) {2};
\node[circle, draw] (d) at (4,6) {1};
\draw[thick] (a)--(b) node [midway, label=right:{$4$}] {};
\draw[thick] (b)--(c) node [midway, label=left:{$1$}] {};
\draw[thick] (c)--(d) node [midway, label=right:{$3$}] {};
\node[] (1) at (2,-1) {1};
\draw[thick] (a)--(1);
\node[] (2) at (1.25,-.75) {2};
\draw[thick] (a)--(2);
\node[] (3) at (1,0) {3};
\draw[thick] (a)--(3);
\node[] (4) at (4.75,1.25) {4};
\draw[thick] (b)--(4);
\node[] (5) at (4.75,2.75) {5};
\draw[thick] (b)--(5);
\node[] (6) at (4.75,6.75) {6};
\draw[thick] (d)--(6);

\node[circle, fill=black, inner sep=1.5pt, label=left:{$p_1$}] (a) at (8,1) {};
\node[circle, fill=black, inner sep=1.5pt, label=right:{$p_2$}] (b) at (12,1) {};
\node[circle, fill=black, inner sep=1.5pt, label=left:{$p_3$}] (c) at (10,4) {};
\draw[very thick] (a)--(b) node (H13) [midway] {};
\draw[very thick] (b)--(c) node (H23) [midway] {};
\draw[very thick] (c)--(a) node (H12) [midway] {};

\draw[thick] (6,1)--(9,5.5) node (A) [midway] {};
\draw[thick] (15.4,-1.1)--(11,5.5) node (B) [midway] {};
\draw[thick] (7.5,-.75)--(15.5,-.75) node (C) [midway] {};
\draw[thick] (8.5,5.25)--(11.5,5.25) node (D) [midway] {};
\draw[thick] (6.25,1.75)--(6.25,-1) node (E) [midway] {};

\node[label=left:{$q_1$}] at (6.25,-.5) {$-$};
\node[label=left:{$q_2$}] at (6.25,0) {$-$};
\node[label=left:{$q_3$}] at (6.25,.5) {$-$};
\node[label=above:{$q_4$}] at (9.6,5.25) {$|$};
\node[label=above:{$q_5$}] at (10.4,5.25) {$|$};
\node[label=below:{$q_6$}] at (8,-.75) {$|$};

\draw[thick] (A) edge[->, bend right] (H12);
\node at (8.45,3.1) {$4:1$};
\draw[thick] (B) edge[->, bend right] (H23);
\node at (12,2.3) {$1:1$};
\draw[thick] (C) edge[->, bend left] (H13);
\node at (11,0) {$3:1$};
\draw[thick] (E) edge[->, bend right] (a);
\draw[thick] (D) edge[->, bend left] (c);
}
\end{center}
\caption{The left-hand image is an example of a localization graph, and the right-hand image is a schematic description of the maps in the corresponding fixed locus. The inner triangle in the right-hand image represents the $\P^2$ inside $\P^4$ spanned by the points $p_1$, $p_2$, and $p_3$. The maps in this fixed locus contract two components of the curve (the one supporting marked points $q_1$, $q_2$, and $q_3$, and the one supporting marked points $q_4$ and $q_5$). The non-contracted components are all mapped onto the corresponding $T$-invariant lines in $\P^4$ as multiple covers of the indicated degree, where any special points are totally ramified over the two corresponding $T$-fixed points in $\P^4$.}\label{fig:localization}
\end{figure}

Let $i_\Gamma:F_\Gamma\hookrightarrow \M_{0,n}(\P^4,d)$ denote the component of the fixed locus indexed by $\Gamma$, and let $G_{n,d}$ denote the set of all decorated graphs that index fixed loci in $\M_{0,n}(\P^4,d)$. Then the localization theorem implies that
\[
\eqref{eq:gwinvariants}=\sum_{\Gamma\in G_{n,d}}\frac{1}{|\mathrm{Aut}(\Gamma)|}\int_{F_\Gamma}\frac{i_\Gamma^*\left(\mathrm{ev}_1^*(\varphi_1)\psi_1^{a_1}\cdots\mathrm{ev}_n^*(\varphi_n)\psi^{a_n}e(R^0\pi_*f^*\O_{\P^4}(5))\right)}{e(N_{F_\Gamma})}.
\]
By normalizing the curves in the fixed locus $F_\Gamma$, the numerator can be computed as a product of terms coming from vertices, edges, and flags (i.e. nodes). The normal bundle has a factor of $R^0\pi_*f^*T\P^4$ from deforming the map $f$ which can also be computed using the normalization sequence, along with a factor from smoothing each node. 

For each edge $e\in E$, let $j_e$ and $j_e'$ denote the labels on the two adjacent vertices. Let $E_v$ denote the set of edges adjacent to $v$, and let $F(\Gamma)$ denote the set of flags $\{(v,e):e\in E_v\}$.  Carefully computing the localization contributions to each fixed locus, one derives the following formula. We refer the reader to \cite[Chapter 27]{MS} for a careful derivation.

\begin{theorem}\label{thm:gwlocalization}
The GW invariant \eqref{eq:gwinvariants} is equal to
\[
\sum_{\Gamma\in G_{n,d}}\frac{1}{|\mathrm{Aut}(\Gamma)|}\prod_{v\in V(\Gamma)}V(j_v,I_v,E_v)\prod_{e\in E(\Gamma)}E(d_e,j_e,j_e')\prod_{(v,e)\in F(\Gamma)}F(j_v)
\]
where
\[
V(j_v,I_v,E_v)=\frac{5\alpha_j}{\prod\limits_{j'\neq j}(\alpha_j-\alpha_{j'})}\int_{\M_{0,\mathrm{val}(v)}}\frac{\prod\limits_{k\in I_v}i_j^*\varphi_k\psi_k^{a_k}}{\prod\limits_{e\in E_v}\left(\frac{\alpha_j-\alpha_{j_e'}}{d_e}-\psi_e \right)},
\]
 \[
E(d,j,j')=\frac{1}{d}\frac{\prod\limits_{k=0}^{5d}\left(5\alpha_j+k\frac{\alpha_{j'}-\alpha_j}{d}\right)}{\prod\limits_{l=0}^4\prod\limits_{k=0 \atop (l,k)\neq (j,0)}^d\left(\alpha_j-\alpha_l+k\frac{\alpha_{j'}-\alpha_j}{d}\right)},
\]
and
\[
F(j)=\left(\frac{5\alpha_j}{\prod\limits_{j'\neq j}(\alpha_j-\alpha_{j'})} \right)^{-1},
\]
where, in the vertex terms, we make special conventions for the unstable moduli spaces:
\[
\int_{\M_{0,2}}\frac{\psi_1^k}{z-\psi_2}=(-z)^k\;\;\;\;\text{ and }\;\;\;\;\int_{\M_{0,1}}\frac{1}{z-\psi_1}=z.
\]
\end{theorem}

If you have never computed $\psi$-class integrals on $\M_{0,n}$, the following exercise is a fun application of the string equation.

\begin{exercise}\label{ex:psi}
Prove that
\[
\int_{\M_{0,n}}\psi_1^{a_1}\cdots\psi_n^{a_n}={n-3 \choose a_1,\dots,a_n}.
\]
\end{exercise}

Along with Theorem \ref{thm:gwlocalization}, Exercise \ref{ex:psi} provides an algorithm to compute any genus-zero GW invariant of the quintic threefold. As a first example, the interested reader can recover the following classical result.

\begin{exercise}\label{ex:2875}
A smooth quintic threefold contains $2875$ lines.
\end{exercise}

All of this seems rather promising; however, the complexity of the graph sums grows at an enormous rate and computations very quickly become inconceivable, even with the most powerful computers at our disposal. If you are not convinced, take a stab at the following exercise.

\begin{exercise}\label{ex:graphs}
Enumerate the $T$-fixed loci in $\M_{0,0}(\P^4,d)$ for $d=2,3,4$.
\end{exercise}

Therefore, in order to prove an explicit result such as Theorem \ref{thm:GWMT}, we need to find and exploit  recursive combinatorial structures in the localization graphs.

\section{Proof of the Gromov-Witten mirror theorem}

In this lecture, we prove the genus-zero GW mirror theorem for the quintic threefold by studying recursions that arise from the combinatorial structure of the localization graphs that were described in the previous lecture.

\begin{proof}[Proof of Theorem \ref{thm:GWMT}]

We begin with the left-hand side of Theorem \ref{thm:GWMT}. Pushing forward to $H^*(\P^4)$, we obtain
\[
i_*J(\tau,z)=5Hz+5H\tau(-z)+\sum_{n,d \atop 0\leq i\leq 4}\frac{q^d}{n!}\left\langle \tau(\psi)^n\frac{H^i}{z-\psi}\right\rangle_{0,n+1,d}^\mathrm{GW}H^{4-i},
\]
where we can safely include the summand $i=4$ because the corresponding invariants vanish. Lifting to equivariant cohomology, then restricting to the fixed point basis, we can write
\[
J_j(\tau,z):=\frac{i_j^*i_*J(\tau,z)}{e(N_{p_j/\P^4})}=\frac{5\alpha_j}{\prod\limits_{j'\neq j}(\alpha_j-\alpha_{j'})}\left(z+\tau_j(-z)\right)+\sum_{n,d}\frac{q^d}{n!}\left\langle \tau(\psi)^n\frac{\rho_j}{z-\psi}\right\rangle_{0,n+1,d}^\mathrm{GW},
\]
where $\tau_j(z)$ is obtained from $\tau(z)$ by replacing $H$ with $\alpha_j$ and $\rho_j=\frac{\prod\limits_{j'\neq j}(H-\alpha_{j'})}{\prod\limits_{j'\neq j}(\alpha_j-\alpha_{j'})}$ denotes the equivariant cohomology class that restricts to the unit at the $j$th fixed point. The correlators in $J_j$ can be computed by the localization formulas of the previous lecture. Each contributing graph has a distinguished vertex $v_0$ that supports the $(n+1)$th marked point. We break the graphs up into two types (see Figure \ref{fig:types}):
\begin{enumerate}
\item Graphs where $\mathrm{val}(v_0)=2$; and
\item Graphs where $\mathrm{val}(v_0)>2$.
\end{enumerate}

\begin{figure}
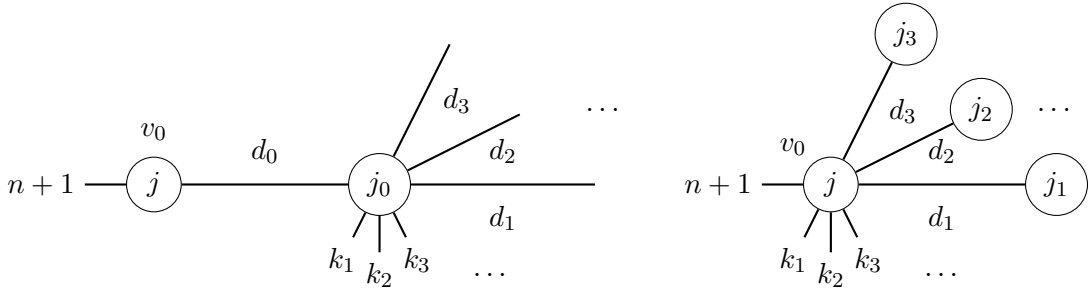

\begin{center}
\tikz{
\node[circle, draw] (a) at (0,0) {$j$};
\node[circle, draw] (b) at (3,0) {$j_0$};
\draw[thick] (a)--(b) node [midway, label=above:{$d_0$}] {};
\node[] (c) at (-1.5,0) {$n+1$};
\draw[thick] (a)--(c);
\node[] at (0,.7) {$v_0$};
\node[] (d) at (6,0) {};
\node[] (e) at (5,1) {};
\node[] (f) at (4,2) {};
\node[] at (6,1) {$\dots$};
\draw[thick] (b)--(d) node [midway, label=below:{$d_1$}] {} node [midway, label=above:{$d_2$}] {};
\draw[thick] (b)--(e);
\draw[thick] (b)--(f) node [midway,label=right:{$d_3$}] {};
\node[] (g) at (2.5,-1) {$k_1$};
\node[] (h) at (3,-1.2) {$k_2$};
\node[] (i) at (3.5,-1) {$k_3$};
\draw[thick] (b)--(g);
\draw[thick] (b)--(h);
\draw[thick] (b)--(i);
\node[] at (4.5,-1.2) {$\dots$};

\node[circle, draw] (b) at (9,0) {$j$};
\node[] at (8.5,.5) {$v_0$};
\node[circle, draw] (d) at (12,0) {$j_1$};
\node[circle, draw] (e) at (11,1) {$j_2$};
\node[circle, draw] (f) at (10,2) {$j_3$};
\node[] at (12,1) {$\dots$};
\node[] (c) at (7.5,0) {$n+1$};
\draw[thick] (b)--(c);
\draw[thick] (b)--(d) node [midway, label=below:{$d_1$}] {} node [midway, label=above:{$d_2$}] {};
\draw[thick] (b)--(e);
\draw[thick] (b)--(f) node [midway,label=right:{$d_3$}] {};
\node[] (g) at (8.5,-1) {$k_1$};
\node[] (h) at (9,-1.2) {$k_2$};
\node[] (i) at (9.5,-1) {$k_3$};
\draw[thick] (b)--(g);
\draw[thick] (b)--(h);
\draw[thick] (b)--(i);
\node[] at (10.5,-1.2) {$\dots$};
}
\end{center}
\caption{Graphs of type (1) (left) and type (2) (right).}\label{fig:types}
\end{figure}

Notice that every graph of type (2) can be obtained in a unique way by taking some number of graphs of type (1), identifying their distinguished vertices as a single vertex, and possibly adding some additional points at the new distinguished vertex. Therefore, away from the vertex $v_0$, contributions from graphs of type (2) are equal to sums of contributions of graphs of type (1). We now investigate what happens at the vertex $v_0$.

For graphs of type (1), the contribution of the vertex $v_0$ is
\begin{equation}\label{eq:type1vert}
\frac{5\alpha_j}{\prod\limits_{j'\neq j}(\alpha_j-\alpha_{j'})}\frac{1}{z+\frac{\alpha_j-\alpha_{j_0}}{d_0}}.
\end{equation}
This expression can be expanded as a power series in $z$. For graphs of type (2), the contribution of the vertex $v_0$ is
\begin{equation}\label{eq:type2vert}
\frac{5\alpha_j}{\prod\limits_{j'\neq j}(\alpha_j-\alpha_{j'})}\int_{\M_{0,\mathrm{val}(v_0)}}\frac{\prod\limits_{k\in I_v}\tau_j(\psi_k)}{\prod\limits_{e\in E_v}\left(\frac{\alpha_j-\alpha_{j_e'}}{d_e}-\psi_e \right)}\frac{1}{z-\psi_{n+1}},
\end{equation}
where the denominator is expanded as a power series in the $\psi$-classes. Therefore, in order to compute the contribution from a graph of type (2), one can compute the contributions from the corresponding graphs of type (1), multiply them each by $F(j)$ to cancel the pre-factor in \eqref{eq:type1vert}, replace each $z$ with the appropriate $-\psi_e$, and then compute the integral over $\M_{0,\mathrm{val}(v_0)}$ as in \eqref{eq:type2vert}.

To summarize the above discussion more succinctly, define the power series $T_j(z)$ by the formula
\[
T_j(-z):=F(j)\cdot\left(\text{contributions of all graphs of type (1)}\right).
\]
Then
\begin{equation}\label{eq:Jj}
J_j(\tau,z)=\frac{5\alpha_j}{\prod\limits_{j'\neq j}(\alpha_j-\alpha_{j'})}\left(z+\tau_j(-z)+T_j(-z)+\sum_{n\geq 2}\frac{1}{n!}\int_{\M_{0,n+1}} \frac{(\tau_j(\psi)+T_j(\psi))^n}{z-\psi_{n+1}}\right),
\end{equation}
where (by \eqref{eq:type1vert}) we can write
\begin{equation}\label{eq:poles}
T_j(-z)=\sum_{j'\neq j \atop d>0}\frac{T_j^{d,j'}}{z-\frac{\alpha_{j'}-\alpha_{j}}{d}}.
\end{equation}
In addition, by removing the vertex $v_0$ and the edge $e_0$ from graphs of type (1), we see that 
\begin{equation}\label{eq:Jjrecursion}
T_j^{d,j'}=q^dF(j)E(d,j,j')F(j')J_{j'}\left(\tau,z=\frac{\alpha_j-\alpha_{j'}}{d}\right).
\end{equation}
Equations \eqref{eq:Jj}, \eqref{eq:poles}, and \eqref{eq:Jjrecursion} determine $J(\tau,z)$ from $\tau(z)$. To see why, notice that the $q^{\leq d}$ terms in \eqref{eq:Jj} are determined by the $q^{\leq d}$ terms in $T_j(z)$ (and $\tau_j(z)$), and the edge-removal recursion \eqref{eq:Jjrecursion} computes $q^{\leq d}$ terms of $T_j(z)$ as $q^{<d}$ terms of $J_j'(\tau,z)$.

Thus, in order to prove Theorem \ref{thm:GWMT},  we must show that
\begin{equation}\label{eq:Ij}
I_j(q,z):=\frac{i_j^*i_*J(\tau,z)}{e(N_{p_j/\P^4})}=z\sum_{d\geq 0}q^d\frac{\prod\limits_{k=0}^{5d}(5\alpha_j+kz)}{\prod\limits_{l=0}^4\prod\limits_{k=0 \atop (l,k)\neq (j,0)}^d(\alpha_j-\alpha_l+kz)}
\end{equation}
can also be written in the form \eqref{eq:Jj} for some $\widetilde T_j^{d,j'}$ that satisfy the edge-removal recursion \eqref{eq:Jjrecursion}. Since \eqref{eq:Ij} has the same simple poles as \eqref{eq:poles}, along with higher-order poles at $z=0$, then the partial fractions decomposition implies that
\[
I_j(q,z)=\frac{5\alpha_j}{\prod\limits_{j'\neq j}(\alpha_j-\alpha_{j'})}\left(z+\tau_j(-z)+\widetilde{T}_j(-z)+\O(z^{-1})\right)
\]
where
\[
\widetilde{T}_j(-z)=\sum_{j'\neq j \atop d>0}\frac{\widetilde{T}_j^{d,j'}}{z-\frac{\alpha_{j'}-\alpha_{j}}{d}}.
\]
Moreover, the next exercise, which is a direct computation, verifies that the $q$-series $\widetilde{T}_j^{d,j'}$ satisfy \eqref{eq:Jjrecursion}.

\begin{exercise}
Prove that
\[
{\mathrm{Res}}_{z=\frac{\alpha_{j'}-\alpha_j}{d}}I_j(q,z)=q^dE(d,j,j')F(j')I_{j'}\left(q,z=\frac{\alpha_{j'}-\alpha_j}{d}\right).
\]
\end{exercise}

\noindent Therefore, the only thing that is left to verify is that $I_j(q,z)$ has the form \eqref{eq:Jj}; in other words, we need to check that
\begin{equation}\label{eq:lastcheck}
I_j(q,z)=\frac{5\alpha_j}{\prod\limits_{j'\neq j}(\alpha_j-\alpha_{j'})}\left(z+\tau_j(-z)+\widetilde{T}_j(-z)+\sum_{n\geq 2}\frac{1}{n!}\int_{\M_{0,n+1}} \frac{(\tau_j(\psi)+\widetilde{T}_j(\psi))^n}{z-\psi_{n+1}}\right).
\end{equation}
This last check follows the exact same arguments as the proof of Theorem \ref{FJRWMT}. For the reader's convenience, we outline the main steps and leave the details as exercises.

Define moduli spaces
\[
\G\M_{0,n,d}=\left\{ (C;q_1,\dots,q_n;D;g)\right\}/\sim
\]
where
\begin{itemize}
\item $(C;q_1,\dots,q_n)$ is a pre-stable, $n$-pointed, genus-zero curve;
\item $D$ is an effective degree-$d$ divisor on $C$, supported away from the marks and nodes,
\item $g:C\rightarrow \P^1$ is a degree-$1$ map; and
\item stability: $\omega_{C,\log}\otimes\O(\epsilon D)\otimes g^*\O_{\P^1}(3)$ is ample for all $\epsilon>0$.
\end{itemize}

Define a $T$-equivariant virtual class:
\[
[\G\M_{0,n,d}]_j^{\vir}:=\frac{e(R^0\pi_*(\O(5\mathcal{D})\otimes\O_{5\alpha_j}))}{e(R^0\pi_*(\oplus_{j'\neq j}\O(\mathcal{D})\otimes\O_{\alpha_j-\alpha_{j'}}))}.
\]
The moduli spaces $\G\M_{0,n,d}$ have a $\C^*$-action, and a distinguished fixed locus $i_{\Gamma_0}:\Gamma_0\hookrightarrow\G\M_{0,n,d}$. Define
\[
\J_j(t,q,z):=-z^2\sum_{n,d\geq 0}\frac{q^d}{n!}\frac{i_{\Gamma_0}^*[\G\M_{0,n+1,d}]_j^{\vir}\cap t(\psi)^n}{e(N_{\Gamma_0})}.
\]
where $t(z)=\sum_{k\geq 0} t_kz^k$. The following two exercises are straightforward modifications of the arguments of Lecture \ref{lecture:proofFJRW}, and they finish the proof of Theorem \ref{thm:GWMT}.

\begin{exercise}
Define 
\[
\widetilde{\tau}_j(z):=I_j(q,-z)+z=\tau_j(z)+\widetilde{T}_j(z).
\]
Prove that 
\[
\J_j(\widetilde{\tau}_j,0,z)=\frac{5\alpha_j}{\prod\limits_{j'\neq j}(\alpha_j-\alpha_{j'})}\left(z+\tau_j(-z)+\widetilde{T}_j(-z)+\sum_{n\geq 2}\frac{1}{n!}\int_{\M_{0,n+1}} \frac{(\tau_j(\psi)+\widetilde{T}_j(\psi))^n}{z-\psi_{n+1}}\right).
\] 
and 
\[
\J_j(0,q,z)=I_j(q,z).
\]
\end{exercise}

\begin{exercise}
Prove that
\[
\J_j(t_0+\widetilde{\tau}_j,0,z)=\J_j(t_0,q,z).
\]
\end{exercise}
\end{proof}

\bibliographystyle{abbrv}
\bibliography{biblio}

\end{document}